\documentclass[a4paper, 12pt, reqno]{amsart}
\usepackage{amssymb, amsfonts, amsmath, bm}
\usepackage[all]{xy}
\usepackage[shortlabels]{enumitem}
\usepackage{hyperref, color}
\usepackage{commath}
\usepackage{esdiff}
\usepackage{graphicx,subfigure}
\usepackage[center]{caption}
\usepackage{amsthm}
\usepackage{diagbox}
\usepackage{multirow}
\usepackage{esdiff}
\usepackage[pagewise]{lineno}
\usepackage[utf8]{inputenc}
\usepackage{soul}
\usepackage{float}
\usepackage{marvosym}
\usepackage{hyperref, color}
\hypersetup{
	colorlinks = true,
	linkcolor = black,
	filecolor = black,
	urlcolor = black,
	citecolor = black
}

\usepackage[left = 2.1 cm, right = 2.1 cm, top = 2.4 cm, bottom = 2.4 cm]{geometry}

\DeclareMathOperator{\sech}{sech}

\newcommand{\Rbb}{\mathbb{R}}
\newcommand{\Sbb}{\mathbb{S}}

\newtheorem{theorem}{Theorem}
\newtheorem{lemma}{Lemma}

\newtheorem{corollary}{Corollary}

\theoremstyle{definition}
\newtheorem{example}{Example}

\newtheorem{remark}{Remark}
\newtheorem{definition}{Definition}
\newtheorem*{acknowledgements}{Acknowledgements}

\numberwithin{figure}{section}

\usepackage{graphicx}

\title[Ruled Ricci surfaces and curves of constant torsion]{Ruled Ricci surfaces and curves of constant torsion}

\author[A. de Carvalho, I. Domingos and R. Santos]{Alcides de Carvalho, Iury Domingos and Roney Santos}

        \address{Universidade Federal de Pernambuco\\
            Departamento de Matem\'atica - CCEN-UFPE\\
            Av. Jornalista An\'ibal Fernandes, S/N, Cidade Universit\'aria,
            50740-560, Recife - PE, Brazil.}	
	\email{alcides.junior@ufpe.br}
 
	\address{Universidade Federal de Alagoas\\
        Av. Manoel Severino Barbosa S/N,
        57309-005, Arapiraca - AL,
        Brazil.}
        \email{iury.domingos@im.ufal.br}
	
	\address{Universidade de S\~ao Paulo\\
		Departamento de Matem\'atica\\
        Rua do Mat\~ao, 05508-900, S\~ao Paulo - SP, Brazil.}
	\email{roneypsantos@ime.usp.br}



\begin{document}

\begin{abstract}
    We show that all non-developable ruled surfaces endowed with Ricci metrics in the three-dimensional Euclidean space may be constructed using curves of constant torsion and its binormal. This allows us to give characterizations of the helicoid as the only surface of this kind that admits a parametrization with plane line of striction, and as the only with constant mean curvature.
\end{abstract}

\maketitle

\section{Introduction}

Although very old, the theory of minimal surfaces remains an active and interesting field in Differential Geometry. These are the surfaces in a Riemannian 3-manifold that are critical points of the area functional. Equivalently, a surface is minimal if its mean curvature vanishes identically.

In 1873, Schl\"afli posed an enduring natural question about isometric immersions that remains unanswered to this day \cite{MR645762, MR1216573}: can every Riemannian surface be locally isometrically embedded in the flat space $\mathbb{R}^3$? A partially affirmative response was achieved in 2003 by Han-Hong \cite{han2003local} contingent upon the behavior of the gradient of the Gauss curvature in the neighborhood of its zeroes.

In 1895, Ricci-Curbastro raised a related question \cite{Ricci-Curbastro}: under what conditions does a Riemannian surface support local minimal isometric embedding in the flat space $\mathbb{R}^3$? He gave a partial answer in the same work for the case of strictly negative Gauss curvature $K$ satisfying $4K = \Delta\log(-K).$

Finally, in 2015, Moroianu-Moroianu \cite{MoroianuRicci} given a full positive answer to the Ricci-Curbastro problem. They showed that a Riemannian surface admits local isometric embedding in $\Rbb^3$ as minimal surface when its Gauss curvature $K$ is non-positive and satisfies
\[K\Delta K - \|\nabla K\|^2 - 4K^3 = 0.\]

The concept of Ricci surface was introduced by Moroianu-Moroianu in their work \cite{MoroianuRicci}. Remember that we call \textit{Ricci surface} a Riemannian surface whose Gauss curvature satisfies the above equation, which is called \textit{Ricci condition}. Furthermore, its metric is referred as \textit{Ricci metric}. Among other results, they proved that $K$ does not change sign, and either $K$ vanishes identically or has only isolated zeroes.

Since Ricci surfaces are an ``intrinsic'' way to see minimal surfaces of $\Rbb^3,$ it is natural the interest in this kind of surfaces. The theory of Ricci surfaces has attracted considerable interest recently and has been developed in several works both from the intrinsic point of view \cite{ambrozio2024intrinsic, daniel2023generalized, MoroianuRicci, zang2022non} and as immersed surfaces in $\Rbb^3$ \cite{domingos2023rotational}.

Our goal in this work is to study ruled Ricci surfaces, namely, ruled surfaces of $\Rbb^3$ whose induced metrics are Ricci. Initially we focus on their metric classification: we aim to identify the conditions under which a certain given family of metrics can be classified as Ricci metrics. Subsequently, we transition to the process of obtaining the immersion of these surfaces. In this direction, we discovered a relation between ruled Ricci surfaces and curves with constant torsion. This correspondence provides a full classification of non-developable ruled Ricci surfaces, and shows that this class of surfaces is as numerous as the class of curves of constant torsion. As consequences of our results, we obtain that the helicoid is the only surface of this kind whose mean curvature is constant, and which may be parametrized with line of striction contained in some plane.

\begin{acknowledgements}
    We are grateful to L. Ambrozio for his valuable comments on this work. A. de Carvalho was partially supported by CNPq grant 150285/2023-0 as well as FAPEAL under the process E:60030.0000000513/2023. I. Domingos was partially supported by the Brazilian National Council for Scientific and Technological Development (CNPq), grant 409513/2023-7. R. Santos was partially supported by grant 2023/14796-3, S\~ao Paulo Research Foundation (FAPESP).
\end{acknowledgements}

\section{Preliminaries}

We remember some properties of ruled surfaces in order to fix notations. Let us consider two curves $\alpha, \beta: I \to \Rbb^3$ such that $\beta(t) \neq 0$ for each $t \in I.$ We call $\Sigma$ a \textit{ruled surface} if it can be parametrized by the immersion $X: I \times J  \to \Rbb^3$ given as
\[X(t,u) = \alpha(t) + u\beta(t).\]
One say that the curves $\{\alpha(t), \beta(t)\}$ \textit{generate} $\Sigma,$ and that $\alpha$ is the \textit{line of striction} corresponding to $\beta$ when they are orthogonal.

Before proceeding, we need to set some properties of the curves generating a ruled surface. As we will see, the non-developable ruled Ricci surfaces we consider admit parametrization whose line of striction has constant non-zero speed. So, let us remember some facts about regular curves in $\Rbb^3.$ For more details, we refer the reader to \cite{manfredo}.

By considering $\alpha$ regular, let us apply a change of parameters $s = s(t)$ where $s$ is the arc length of $\alpha$ and reparametrize $\Sigma$ as
    \[X'(s,u) = \alpha(s) + u\beta(s),\ (s,u) \in I' \times J.\]
    For our purpose, let us assume for now that $\alpha$ is not necessarily orthogonal to $\beta.$ At each point $\alpha(s),$ with $\alpha^{\prime\prime}(s)\neq 0,$ we have associated an orthonormal frame $\{T(s), N(s), B(s)\},$ called \textit{Frenet frame}, that is formed respectively by the unit tangent, the unit normal and the unit binormal vectors of $\alpha$ at $\alpha(s).$ More precisely, we have $\alpha'(s) = T(s),$ $\alpha''(s) = \kappa(s)N(s)$ for a function $\kappa: I' \to \Rbb$ called \textit{curvature} of $\alpha$ and $B(s) = T(s) \wedge N(s).$ These vector fields obey the well-known \textit{Frenet formulas}
    \[T' = \kappa N,\ \ \ N' = -{\kappa T - \tau B},\ \ \ B' = \tau N.\]
    We have omitted the parameter $s$ and the prime means the derivative with respect to $s.$ The function $\tau: I' \to \Rbb$ is called \textit{torsion} of $\alpha.$ For our objective, let's assume that the zeros of the curvature of $\alpha$ are isolated.

    A standard way to change $\alpha$ with the line of striction $\widetilde{\alpha}$ of $\beta$ is as follows. We write $\beta = \beta_1T + \beta_2N + \beta_3B$ and consider
    \[\widetilde{\alpha}(s) = \alpha(s) - h(s)\beta(s),\]
    for some function $h:I' \to \Rbb$ which we want to find. By admitting $\langle \widetilde{\alpha}', \beta'\rangle = 0$ in $I',$ a straightforward computation using Frenet formulas gives that
    \[h(s) = \frac{\beta_1'(s) - \kappa(s)\beta_2(s)}{\|\beta'(s)\|^2}.\]
    In particular, if the original curve $\alpha$ is the line of striction of $\beta,$ we must have
    \begin{equation}\label{beta}
        \beta_1'(s) = \kappa(s) \beta_2(s).
    \end{equation}

From now on, we will not assume that $\alpha$ is necessarily regular. Regarding the curve $\beta,$ we begin by noting that it can be assumed, without loss of generality, that its trace is contained in $\Sbb^2$ and that $\alpha$ is its line of striction (see \cite{manfredo}). Moreover, to be able to develop the theory, we need the non-trivial assumption that $\beta'$ has no zero in $I.$ In the case where the zeroes of \( \beta'\) are isolated, we can divide our surface into pieces such that the theory can be applied to each of them. However, if the zeroes of \( \beta' \) have cluster points, the situation may become complicated and we will avoid it here. The assumption that $\beta$ is a regular curve in $I$ is usually expressed by saying that $\Sigma$ is {\it non-cylindrical.}

Under all considerations above, unless otherwise stated, from now on we shall assume that
\[
X(t, u) = \alpha(t) + u\beta(t)
\]
is a parametrization of a non-cylindrical ruled surface whose generating curves are orthogonal and \( \|\beta(t)\| = \|\beta'(t)\| = 1 \) for all \( t \in I \). The induced metric in this setting is expressed as
\[g = \left(\|\alpha'(t)\|^2 + u^2\right)dt^2 + 2\langle \alpha'(t), \beta(t) \rangle dtdu +du^2.\]
The assumption that $\beta$ is a spherical curve implies that \( \langle \beta(t), \beta'(t) \rangle = 0 \) for all \( t \in I \). Therefore, we get $\langle \alpha'(t), \beta'(t) \rangle = \langle \beta(t), \beta'(t) \rangle = 0,$ and consequently there exists a function $\lambda: I \to \Rbb^3$ such that $\alpha'(t) \wedge \beta(t) = \lambda(t) \beta'(t),$ where $\wedge$ denotes the cross product. This function $\lambda$ is called \textit{distribuition parameter} of $X.$ Since
\[\|\alpha'(t)\|^2 + u^2 - \langle \alpha'(t), \beta(t)\rangle^2 = \|X_t(t,u) \wedge X_u(t,u)\|^2 = \lambda(t)^2 + u^2,\]
without loss of generality $g$ can be rewritten as
\begin{equation}\label{ruledmetric}
g = \left(\|\alpha'(t)\|^2 + u^2\right)dt^2 + 2\sqrt{\|\alpha'(t)\|^2 - \lambda(t)^2} dtdu +du^2.
\end{equation}

We say that a ruled surface is \textit{developable} if its distribution parameter vanishes identicaly. Otherwise, we say that the ruled surface is \textit{non-developable}. The relation between the Gauss curvature $K$ and the distribution parameter $\lambda$ of $X$ is given by the formula
\[K(t,u) = -\frac{\lambda(t)^2}{\left(\lambda(t)^2 + u^2\right)^2}.\]
In particular, the distribution parameter of an immersion of a ruled Ricci surface vanishes either identically or at isolated points due to the property on the zeroes of the Gauss curvature of a Ricci surface. The first case trivially implies that the considered ruled surface must be a Ricci surface. For this reason, we will focus only on the case of non-developable ruled Ricci surfaces.

\section{Ruled Ricci surfaces}

As we will see, the Ricci condition imposes strong restrictions to a certain class of metrics that generalize the one of ruled surfaces. Precisely, on the case of ruled surfaces, we will show that the distribution parameter and the speed of the line of striction of a given parametrization of a ruled Ricci surface are constant and equal, up to sign.

\subsection{Metric classification}\label{metricclassificationsection}

In this section we will obtain necessary and sufficient condition for a class of metrics that have as particular cases the metrics induced by ruled surfaces to be Ricci metrics.

\begin{lemma}\label{metricclassification}
    Let $g$ be a non-flat Riemannian metric given as
    \[g = \left(f(t)^2 + u^2\right)dt^2 + 2\sqrt{f(t)^2 - \lambda(t)^2}dtdu + du^2,\ (t,u) \in I \times J,\]
    where $f, \lambda: I \to \Rbb$ are functions with $f(t)^2 \geq \lambda(t)^2$. Then, $g$ is a Ricci metric if and only if $\lambda(t)^2 = f(t)^2 = c^2$ for some $c>0$ and each $t \in I.$ In particular, we have $g = (c^2+u^2)dt^2 + du^2.$
\end{lemma}

\begin{proof}

Suppose $g$ is a Ricci metric. Firstly, we will show $\lambda(t)^2 = f(t)^2$ in $I.$ To see this, let us admit that $\lambda(t_0)^2 \neq f(t_0)^2$ for some $t_0 \in I.$ In this case, by continuity we must have $\lambda(t)^2 \neq f(t)^2$ at some interval $I_0 \subset I$ containing $t_0.$ From now on, for simplicity we will omit the parameter $t$ and use a prime to denote the derivative with respect to $t.$ We restrict ourselves to an open set $U\subset I_0\times J$ on which $K<0$. Thus, the Ricci condition is equivalent to $\Delta\log(-K) = 4K$ on $U$. Notice that the curvature $K$ of $g$ is given as
    \[K = -\frac{\lambda^2}{\left(\lambda^2 + u^2\right)^2}.\]
Since
\[\Delta\log(\lambda^2) = \frac{2 {\left({\left(\lambda \lambda'' - (\lambda')^2\right)} u^2 + \delta \lambda \lambda'u + \lambda^2\left(\lambda\lambda'' - 2(\lambda')^2\right)\right)}}{\lambda^2(\lambda^2 + u^2)^2}\]
and
\[\Delta\log(\lambda^2 + u^2)^2 = -\tfrac{4\left({\left(f f' - \lambda \lambda'\right)} u^3 + \left(2 f^2 - 3  \lambda^2 - (\lambda')^2 - \lambda \lambda''\right)\delta u^2 + \lambda\left(\lambda f f' - \left(6  f^2 - 5  \lambda^2\right) \lambda'\right) u - \lambda^2\left(f^2 + \lambda \lambda'' - 2 (\lambda')^2\right)\delta\right)}{(\lambda^2 + u^2)^3\delta},\]
where $\delta(t) = \sqrt{f(t)^2 - \lambda(t)^2}$, we get that the Ricci condition on $U$ is equivalent to
\begin{equation}\label{polynomial}
    c_4u^4 + c_3u^3 + c_2u^2 + c_1u + c_0 = 0,
\end{equation}
for $c_i:I_0 \to \Rbb,$ $i=0, 1, 2, 3, 4,$ given as
\begin{align*}
    c_4 &= \delta\left(\lambda \lambda'' - (\lambda')^2\right)\\
 c_3 &= \lambda\left(-3\lambda^2\lambda' + f^2\lambda' + \lambda (f^2)'\right)\\
    c_2 &= \delta\lambda^2\left(-4\lambda^2 + 4f^2 - 5(\lambda')^2\right)\\
    c_1 &= \lambda^3\left(9\lambda^2\lambda' - 11f^2\lambda' + \lambda (f^2)'\right)\\
    c_0 &= \delta\lambda^4\left(2\lambda^2 -2f^2 + 2(\lambda')^2 - \lambda\lambda''\right).
\end{align*}
    Note that the left side of equation \eqref{polynomial} is a polynomial with respect to $u.$ Thus, since the coefficients depend only on $t$, we get $c_i = 0$ for $i=0,1,2,3,4.$ By $c_3 = c_1 = 0,$ we get
    \[0 = \left(9\lambda^2\lambda' - 11f^2\lambda' + \lambda (f^2)'\right) - \left(-3\lambda^2\lambda' + f^2\lambda' + \lambda (f^2)'\right) = -12\delta^2\lambda'.\]
    Hence, $\lambda''= \lambda' = 0$ and $c_0 = -4\delta^3\lambda^4 \neq 0$ in $I_0.$ This contradiction shows that $\delta = 0$ in $I.$ Consequently, the metric $g$ is given as
    \[g = \left(\lambda(t)^2 + u^2\right)dt^2 + du^2\]
    and for this metric we get
\[\Delta\log(\lambda^2) = \frac{2\left(\left(\lambda \lambda'' - (\lambda')^2\right) u^2 + \lambda^2\left(\lambda \lambda'' - 2(\lambda')^2\right)\right)}{\lambda^2(\lambda^2 + u^2)^2}\]
and
\[\Delta\log\left(\lambda^2+u^2\right)^2=\frac{4\left(\left(\lambda^2 + \left(\lambda\lambda'\right)'\right) u^2+\lambda^2\left(\lambda^2 - 2\left(\lambda'\right)^2 + \lambda \lambda'' \right)\right)}{(\lambda^2+u^2)^3}.\]   
Once $\lambda \neq 0$ almost everywhere, the Ricci condition in this case is equivalent to
    \begin{equation}\label{polynomial2}
        \left((\lambda')^2 - \lambda\lambda''\right)u^4 + 5\lambda^2(\lambda')^2u^2 - \lambda^4\left(2(\lambda')^2 - \lambda\lambda''\right) = 0.
    \end{equation}
   Since the left hand of equation \eqref{polynomial2} is a polynomial in $u$ whose coefficients are functions of $t,$ as before we conclude that these coefficients vanish. Therefore, the term of second order of \eqref{polynomial2} implies that $\lambda'$ is zero in $I,$ which shows that $\lambda$ is constant. This provides a solution of \eqref{polynomial2} and shows that $\lambda^2 = f^2$ is constant as claimed.

    Reciprocally, assuming $\lambda^2 = f^2$ constant in $I,$ we want to prove that the Ricci condition holds. This is equivalent to show that \eqref{polynomial2} holds, and this can be seen directly.
\end{proof}

\begin{remark}
     An interesting problem in classical Differential Geometry is to know when two surfaces coincide, up to rigid motion of the space. In this direction, the only minimal isometric immersions of $\mathbb{R}^2$, endowed with the Ricci metric of the above result, into $\mathbb{R}^3$ are the ones congruent to the one-parameter associated family with the helicoid (cf. \cite[Theorem 8]{Lawson}).

\end{remark}

\subsection{Ricci condition for ruled surfaces}

A direct consequence of the metric classification in Lemma \ref{metricclassification} says that the distribution parameter of a certain parametrization of a ruled Ricci surface is constant and equals to the speed of its line of striction, up to sign.

\begin{corollary}\label{distributionparameter}
    Let $\Sigma$ be a non-developable ruled surface parametrizated by $X: I \times J \to \Rbb^3$ as
    \[X(t,u) = \alpha(t) + u \beta(t)\]
    such that the curves $\alpha: I \to \Rbb^3$ and $\beta: I \to \Sbb^2$ are orthogonal and $\beta$ is parametrized by arc length. Then, $\Sigma$ is a Ricci surface if and only $\lambda(t)^2 = \|\alpha'(t)\|^2 = c^2$ for some non-zero real constant $c$ and each $t \in I,$ and where $\lambda$ is the distribution parameter of $X.$ In particular, the induced metric of $\Sigma$ is $g = (c^2+u^2)dt^2 + du^2.$
\end{corollary}

\begin{proof}
    Remember that the induced metric $g$ of $\Sigma$ is written as in \eqref{ruledmetric}. The result follows by considering, in Lemma \ref{metricclassification}, $f(t) = \|\alpha'(t)\|$ and $\lambda$ as the distribution parameter of $X.$
\end{proof}

\begin{remark}
    We point out that in the previous notations for ruled surfaces, since we have $\langle \alpha'(t), \beta(t)\rangle^2 = \|\alpha'(t)\|^2 - \lambda(t)^2,$ the condition stated in Corollary \ref{distributionparameter} for a ruled Ricci surface implies that $\langle \alpha'(t), \beta(t)\rangle$ vanishes in $I.$ However, $\|\alpha'(t)\|^2 = \lambda(t)^2$ in $I$ does not necessarily imply that the surface is a ruled Ricci surface. Indeed, to see this let us consider the following parametrization of a family of ruled surfaces given by
    \[X_w(t,u) = (0, 0, w(t)) + u(\cos(t), \sin(t), 0),\]
    where $w:I \to \Rbb$ is a smooth function. We have $\alpha(t) = (0, 0, w(t))$ and $\beta(t) = (\cos(t), \sin(t), 0).$ Obviously, $\alpha$ is the line of striction of $\beta$ and $\|\beta(t)\| = \|\beta'(t)\| = 1$ for any $t \in I.$ Moreover, $\langle \alpha'(t), \beta(t)\rangle = 0$ and therefore $\lambda(t)^2 = \|\alpha'(t)\|^2$ in $I.$ But $X_w$ is not the parametrization of a ruled Ricci surface in general. In fact, if we admit that $w$ is not an affine function, in this case, $\|\alpha'\|$ (and consequently the distribution parameter of $X_w$) is not constant and, according to Corollary \ref{distributionparameter}, $X_w$ is not a parametrization of a ruled Ricci surface.
    
 This class is known as right conoids and includes helicoids, hyperbolic paraboloids, Plücker conoids, Wallis conicals edges and Whitney umbrellas. As we will see, except for the helicoid, none of them are Ricci surfaces.

\end{remark}

    We will use Corollary \ref{distributionparameter} in order to classify ruled Ricci surfaces in two aspects depending on the geometry of the line of striction for a certain parametrization. For the first one, we consider that the surface may be parametrized in such a way that the line of striction is contained in a straight line. Remember that the immersion
    \[X(t,u) = (0, 0, at + b) + u (\cos(t), \sin(t), 0),\ (t,u) \in I \times J,\]
    for any constants $a$ and $b$ with $a \neq 0$ gives a parametrization of a piece of helicoid. We have the following assertion.

    \begin{theorem}\label{classification1}
    Let $\Sigma$ be a non-developable ruled Ricci surface which can be parametrized so that its line of striction is contained in a straight line. Then, $\Sigma$ is a piece of helicoid.
    \end{theorem}
    
    \begin{proof}
    After rigid motions if necessary, we can parametrize $\Sigma$ as
    \[X(t,u) = (0, 0, w(t)) + u(x(t), y(t), z(t)),\ (t,u) \in I \times J,\]
    where $w, x, y, z: I \to \Rbb$ are smooth functions. We can assume that these functions satisfy the conditions $x^2 + y^2 + z^2 = 1$ and $ (x')^2 + (y')^2 + (z')^2 = 1,$ and furthermore $(w')^2 = c^2$ for some positive constant $c$ according to Corollary \ref{distributionparameter}. In particular, we have $w(t) = ct + a$ for some constant $a,$ and once $cz = \langle (0, 0, w'), (x, y, z)\rangle = 0$ then $z = 0$. Therefore, the parametrization $X$ becomes
    \[X(t,u) = (0, 0, ct + a) + u(x(t), y(t), 0),\]
    and consequently its distribution parameter $\lambda$ is given by
    \[\lambda^2 = \langle (0, 0, w') \wedge (x, y, 0), (x', y', 0)\rangle^2 = c^2(xy'- yx')^2.\]
    According to Corollary \ref{distributionparameter}, we have $\lambda^2 = c^2$, which implies $xy' - yx' = 1,$ up to sign. Since $x^2 + y^2 = 1$, we must have $x(t) = \cos\theta(t)$ and $y(t) = \sin\theta(t)$ for some function $\theta: I \to \mathbb{R}$, and thus
    \[1 = xy' - yx' = \cos^2(\theta)\theta'+\sin^2(\theta)\theta' = \left(\cos^2(\theta)+\sin^2(\theta)\right)\theta'=\theta'.\]
    Hence, $\theta(t) = t + b$ for some constant $b,$ which proves our assertion.
\end{proof}

For the second classification result, let us assume that the ruled surface that we consider does not admit parametrization with straight line of striction. We obtain that the surface is generated by a constant torsion curve and its binormal. More precisely, we have the following result.

    \begin{theorem}\label{classification2}
    Let $\Sigma$ be a non-developable ruled surface parametrized by $X:I \times J \to \Rbb^3$ as
    \[X(t,u) = \alpha(t) + u \beta(t)\]
    such that the curves $\alpha: I \to \Rbb^3$ and $\beta: I \to \Sbb^2$ are orthogonal and $\beta$ is parametrized by arc length. In addition, suppose that the curvature of $\alpha$ has no zero in $I.$ Then, $\Sigma$ is a Ricci surface if and only if $\alpha$ has constant torsion $\tau_0 \neq 0$ and $\beta$ is its binormal. Moreover, the constant distribution parameter $\lambda_0$ of $X$ is such that $(\lambda_0\tau_0)^2 = 1.$
    \end{theorem}

\begin{proof}
    Suppose $\Sigma$ is a Ricci surface. First, we consider a parametrization $X: I \times J \to \Rbb^3$ of $\Sigma$ given by $X(t,u) = \alpha(t) + u\beta(t)$ for orthogonal curves $\alpha, \beta: I \to \Rbb^3$ with $\|\beta(t)\| = \|\frac{d\beta}{dt}\| = 1.$ Initially, we show that after a reparametrization of $\alpha$ by arc length the surface is generated by $\alpha$ and its binormal. We know by means of Corollary \ref{distributionparameter} that $\alpha$ is regular because $\|\frac{d\alpha}{dt}\|^2 = c^2$ for some $c>0,$ and thus we can reparametrize it by its arc length $s = s(t).$ Without loss of generality, we admit $0 \in I,$ and thus
    \[s(t) = \int_0^t \left\|\frac{d\alpha}{d\sigma}\right\|d\sigma = ct.\]
    Let $(\beta_1, \beta_2, \beta_3)$ be the coordinate functions of $\beta$ with respect to the Frenet frame $\{T, N, B\}$ of $\alpha.$ Since by Corollary \ref{distributionparameter}
    \[0 = \bigg\langle \frac{d\alpha}{ds}, \beta\bigg\rangle = \frac{1}{c} \langle T, \beta_1T + \beta_2N + \beta_3B\rangle = \frac{\beta_1}{c}\]
    one get $\beta_1 = 0.$ Furthermore, since the curvature of $\alpha$ is different from zero everywhere it follows from equation \eqref{beta} that $\beta_2$ vanishes along $I.$ Finally, since the trace of $\beta$ is contained in $\Sbb^2,$ one conclude that $\beta_3^2 = 1.$ This shows that the immersion $X':I'\times J' \to \Rbb^3$ given as
    \[X'(s,u) = \alpha(s) + uB(s)\]
    is a reparametrization of $\Sigma.$ Now, we assert that the torsion $\tau$ of $\alpha$ is constant. Indeed, using the Frenet formulas, we obtain
    \[c^2 = \lambda^2 = \bigg\langle \frac{d\alpha}{dt} \wedge B, \frac{dB}{dt}\bigg\rangle^2 = c^4\tau^2 \langle T \wedge B, N \rangle^2 = c^4\tau^2,\]
    which implies that $\tau^2 = 1/c^2.$

    Conversely, let us suppose that $\Sigma$ is parametrized by an immersion $X: I\times J \to \Rbb^3$ as
    \[X(s,u) = \alpha(s) + uB(s),\ \ \ (s,u) \in I \times J,\]
    where $\alpha: I \to \Rbb^3$ is a regular curve with constant torsion $\tau_0 \neq 0,$ arc length $s$ and binormal $B: I \to \Sbb^2.$ In this case, the arc length $t=t(s)$ of $B$ is given by
    \[t(s) = \int_0^s \left\|\frac{dB}{d\sigma}\right\|d\sigma = |\tau_0|s,\]
    where, whitout loss of generality, we have admited $0 \in I.$ On one hand, we have $\langle \frac{d\alpha}{dt}, B \rangle = 0,$ and consequently $\lambda^2 = \|\frac{d\alpha}{dt}\|^2,$ where $\lambda$ is the distribution parameter of $X.$ On the other hand,
    \[\lambda^2 = \left\langle \frac{d\alpha}{dt} \wedge B, \frac{dB}{dt}\right\rangle^2 = \frac{1}{\tau_0^2}\left\langle T\wedge B, N\right\rangle^2 = \frac{1}{\tau_0^2},\]
    where $\{T, N, B\}$ is the Frenet frame of $\alpha.$ Hence, Corollary \ref{distributionparameter} implies that $\Sigma$ is a Ricci surface.
\end{proof}

\begin{remark}
We point out that Theorems \ref{classification1} and \ref{classification2} suggest a correspondence between curves of constant torsion and non-developable ruled Ricci surfaces. Furthermore, these results are a classification for non-developable ruled Ricci surfaces in the sense that they are a way to obtain this kind of surfaces.
\end{remark}

    By combining our results with the fact that plane curves are characterized as the ones with vanishing torsion, we get the following.

    \begin{corollary}
        The only non-developable ruled Ricci surface which may be parametrized with line of striction contained in a plane is the helicoid.
    \end{corollary}

    Note that, in particular, we have proven that all non-developable ruled Ricci surfaces may be generated by a constant torsion curve and its binormal. In Theorem \ref{classification1} the line of striction is a straight line and its binormal is a great circle, while in Theorem \ref{classification2} the line of striction has contant non-zero torsion and its binormal is not contained in a great circle. In both cases, the spherical curve is the binormal of the line of striction. Furthermore, since the spherical curve is regular, it is direct to see that Theorem \ref{classification2} does not depend on its arc length parameter. This motivates the following.

\begin{definition}
     We call \textit{canonical parametrization} a parametrization of a non-developable ruled Ricci surface generated by $\{\alpha(t), B(t)\}$ where $\alpha$ has constant torsion and $B$ is the binormal of $\alpha.$
\end{definition}

\begin{remark}
Let $\alpha:I \to \Rbb^3$ be a regular curve parametrized by arc length, and let $\{T, N, B\}$ be its Frenet frame. The \textit{tangent surface} and the \textit{normal surface} generated respectively by $\{\alpha(s), T(s)\}$ and by $\{\alpha(s), N(s)\}$ are developable, which means that they have zero Gauss curvature. It follows immediately that they are ruled Ricci surfaces. Furthermore, our Theorem \ref{classification2} says that the \textit{binormal surface} generated by $\{\alpha(s), B(s)\}$ is a canonical parametrization of a non-developable ruled Ricci surface provided that $\alpha$ has constant torsion.
\end{remark}

    The next result relates the mean curvature of a non-developable ruled Ricci surface with the curvature of the line of striction of a canonical parametrization.

    \begin{theorem}
        The mean curvature of a non-developable ruled Ricci surface $\Sigma$ is
        \[H(t,u) = -\frac{\kappa(t)}{2\sqrt{1 + \tau_0^2u^2}},\]
        where $\kappa$ and $\tau_0$ are respectively the curvature and the torsion of the line of striction in the canonical parametrization of $\Sigma.$ In particular, the only non-developable ruled Ricci surface with constant mean curvature is the helicoid.
    \end{theorem}

    \begin{proof}
        If $\Sigma$ is a piece of helicoid, there is noting to do. For the remaining case, the proof is a long calculation, and for simplicity we will omit the parameters. By considering the immersion $X:I \times J \to \Rbb^3$ of the ruled Ricci surface $\Sigma$ given by Theorem \ref{classification2}
        \[X(s,u) = \alpha(s) + u B(s),\]
      where $s$ and $B$ are the arc length and the binormal of $\alpha,$ respectively, note that in this parametrization the mean curvature of $\Sigma$ is given as
      \[H = \frac{1}{2}\left(\frac{\langle X_{ss}, N_\Sigma\rangle}{\langle X_s, X_s\rangle} + \langle X_{uu}, N_\Sigma\rangle\right),\]
      where $N_\Sigma$ is the unit normal of $\Sigma.$ If we denote respectively by $\{T, N, B\}$ and $\tau_0$ the Frenet frame and the torsion of $\alpha$ we have
      \[X_s = T + \tau_0uN\ \ \mbox{and}\ \ X_u = B.\]
      Direct computations show that
      \[\langle X_s, X_s \rangle = 1 + \tau_0^2u^2,\ \ X_{ss} = \kappa N - \tau_0 u(\kappa T + \tau_0 B)\ \ \mbox{and}\ \ X_{uu} = 0,\]
      where $\kappa$ is the curvature of $\alpha,$ and that $N_\Sigma$ is given by
      \[N_\Sigma = \frac{X_s \wedge X_u}{\|X_s \wedge X_u\|} = \frac{-N + \tau_0 u T}{\sqrt{1 + \tau_0^2u^2}}.\]
      Consequently, $\langle X_{ss}, N_\Sigma \rangle = - \kappa\sqrt{1+\tau_0^2u^2},$
      which implies that
      \[H = -\frac{\kappa}{2\sqrt{1 + \tau_0^2u^2}}.\]
      Since $\kappa$ does not depend on the parameter of the curve, this shows the formula we stated.
    \end{proof}

\subsection{Constructing ruled Ricci surfaces}

 It is well-known that for a given curve $B: I \to \Sbb^2$ parametrized by arc length and a constant non-zero $\tau_0$ there exists a curve $\alpha: I \to \Rbb^3$ given by
\begin{equation}\label{constanttorsion}
    \alpha(t) = \frac{1}{\tau_0}\int_{t_0}^t B'(\sigma) \wedge B(\sigma)d\sigma
\end{equation}
with constant torsion equals $\tau_0$ provided that $\langle B(t) \wedge B'(t), B''(t)\rangle \neq 0,$ and moreover $B$ is the binormal of $\alpha.$ Conversely, for a given curve with constant non-zero torsion, it may be parametrized by formula \eqref{constanttorsion}. We refer the reader to \cite{MR3089777} and references therein for details.

We present here explicit examples of ruled Ricci surfaces illustrating the construction above.

\begin{example}[Parallel circles] \label{parallelcircles}

The first family we will look is associated with closed curves. Let us consider a family of parallel circles in $\Sbb^2$ of radius $\ell \in (0,1)$ centered at $(0, 0, \sqrt{1-\ell^2}),$ which is parameterized by arc length by the map $B_\ell: [0,2\ell\pi]\to \mathbb{S}^2$ defined as
$$B_\ell(t)=\left(\ell \sin\frac{t}{\ell}, -\ell \cos\frac{t}{\ell}, \sqrt{1-\ell^2}\right).$$
Associated to this family we produce a family of curves with constant torsion equals $1.$ A straightforward calculation using \eqref{constanttorsion} shows that the family of circles $B_\ell$ is associated with the family of $\ell$-helix of constant torsion equals $1$ given by
$$
\alpha_\ell(t)=-\ell\left(\sqrt{1-\ell^2}\cos \dfrac{t}{\ell}, \sqrt{1-\ell^2}\sin \dfrac{t}{\ell}, t\right),
$$
up to translation. Hence, Theorem \ref{classification2} gives that the map $X_\ell: [0, 2\ell\pi] \times J \to \Rbb^3,$ given as
$X_\ell(t,u) = \alpha_\ell(t) + u B_\ell(t)$, is the canonical parametrization of a non-developable ruled Ricci surface for each $\ell \in (0, 1).$ One can see directly that the unit circle in $\Sbb^2$ is associated with a helicoid by making $\ell$ goes to 1, and that the family degenerates in the vertical straight line when $\ell$ goes to zero.

\begin{figure}[!ht]
    \begin{subfigure}
        \centering
        \includegraphics[width=0.32\linewidth]{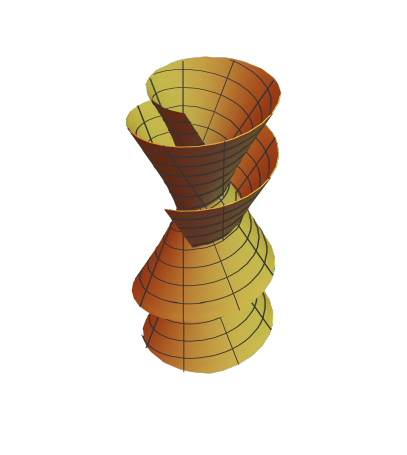}
    \end{subfigure}
    \begin{subfigure}
        \centering
        \includegraphics[width=0.31\linewidth]{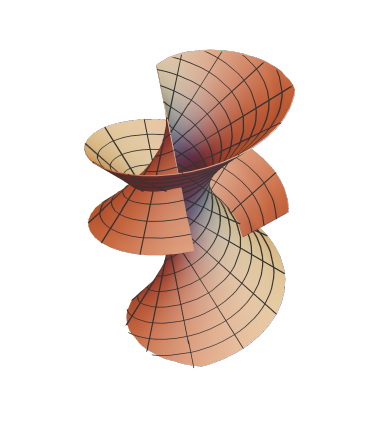}
    \end{subfigure}
    \begin{subfigure}
        \centering
        \includegraphics[width=0.3\linewidth]{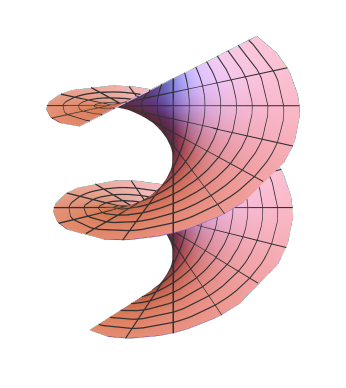}
    \end{subfigure}
    \caption{Non-developable ruled Ricci surfaces associated with circles in $\Sbb^2$ of radius $\frac{1}{2}$, $\frac{3}{4}$ and $1$, respectively.}
    \label{example1}
\end{figure}
\end{example}

\begin{example}[Anti-Salkowski curve]

Now, we will describe a family of ruled Ricci surfaces whose line of striction is defined at most in an open bounded interval. In \cite{salkowski1909zur}, Salkowski constructs a family of curves of constant curvature but non-constant torsion. Later, Monterde demonstrated in \cite{monterde2009salkowski} how to construct, from a curve of constant curvature, another curve of constant torsion. Applying this technique to the family of curves defined by Salkowski, he was able to generate a new family of curves with constant torsion equals $1$ but non-constant curvature, and he referred to this family as anti-Salkowski curves. It may be parametrized by arc length by $\alpha_\ell: \big(-\frac{1}{\ell}, \frac{1}{\ell}\big) \to \Rbb^3$ as
\[\alpha_\ell(t) = 
\begin{bmatrix}
\cos\theta(t) & -\sin\theta(t) & 0\\
\sin\theta(t) & \cos\theta(t) & 0\\
0 & 0 & -1
\end{bmatrix}
\begin{bmatrix}
ca(t)\\
cb(t)\\
\displaystyle\frac{2 \arcsin (\ell t) + \sin (2 \arcsin (\ell t))}{4 \ell \sqrt{1 + \ell^2}}
\end{bmatrix}\]
for
\begin{center}
\begin{tabular}{ll}
 $a(t) = -2 + 3 \ell^2 t^2 + 3 \ell^4 t^2,$ & \ \ \ $\displaystyle c = \frac{\ell}{1 - 2 \ell^2 - 3 \ell^4},$ \\[10pt]
  $b(t) = \sqrt{1 + \ell^2} \left( 1 + 3 \ell^2 \right) t \sqrt{1 - \ell^2 t^2},$ & \ \ \ $\displaystyle \theta(t) = \frac{\sqrt{1 + \ell^2} \arcsin (\ell t)}{\ell},$
\end{tabular}
\end{center}
where $\ell$ is positive and different of $\frac{1}{\sqrt{3}}.$ Its binormal is
\[B_\ell(t) =
\frac{1}{\sqrt{1+\ell^2}}
\begin{bmatrix}
\cos\theta(t) & -\sin\theta(t) & 0\\
\sin\theta(t) & \cos\theta(t) & 0\\
0 & 0 & 1
\end{bmatrix}
\begin{bmatrix}
d(t)\\
\ell^2 t\\
\ell t
\end{bmatrix}\]
for
\[d(t) = -\sqrt{1 + \ell^2} \sqrt{1 - \ell^2 t^2}.\]
Since the torsion of $\alpha$ is $1,$ then $t$ is also the arc length of $B_\ell$. Hence, from Theorem \ref{classification2} the map $X_\ell: \big(-\frac{1}{\ell}, \frac{1}{\ell}\big) \times J \to \Rbb^3,$ given as $X_\ell(t,u) = \alpha_\ell(t) + u B_\ell(t),$ is the canonical parametrization of a non-developable ruled Ricci surface. These surfaces are non-complete since $t$ lies at most in bounded interval. Furhtermore, once $\theta(t) \to t$ as $\ell \to 0,$ we must have that $B_\ell(t)$ becomes a geodesic of $\Sbb^2$ when $\ell$ goes to zero, which means that the family converges to the helicoid in this case. 

\begin{figure}[!ht]
    \begin{subfigure}
        \centering
        \includegraphics[width=0.3\linewidth]{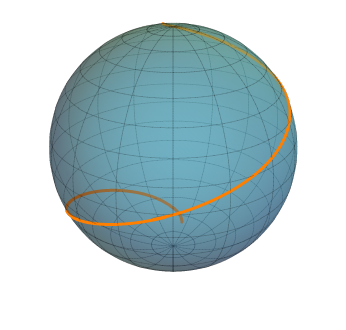}
    \end{subfigure}
    \begin{subfigure}
        \centering
        \includegraphics[width=0.3\linewidth]{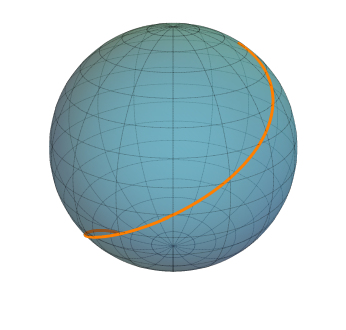}
    \end{subfigure}
    \begin{subfigure}
        \centering
        \includegraphics[width=0.3\linewidth]{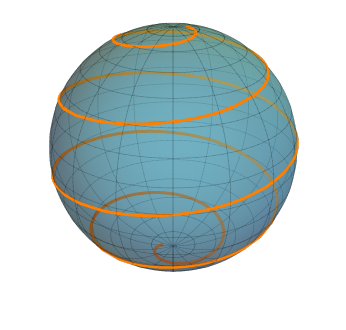}
    \end{subfigure}
    \begin{subfigure}
        \centering
        \includegraphics[width=0.32\linewidth]{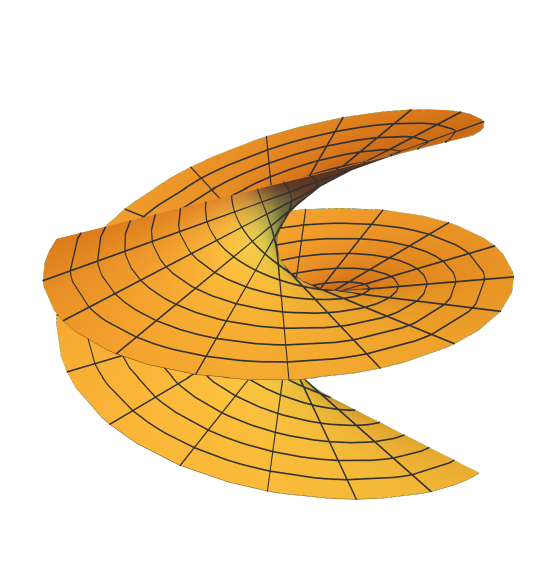}
    \end{subfigure}
    \begin{subfigure}
        \centering
        \includegraphics[width=0.3\linewidth]{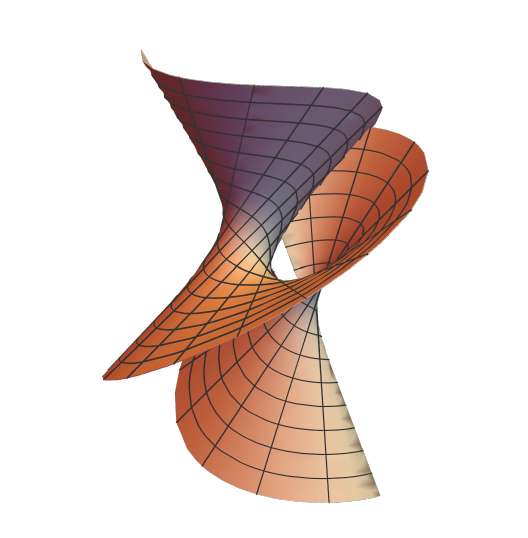}
    \end{subfigure}
    \begin{subfigure}
        \centering
        \includegraphics[width=0.3\linewidth]{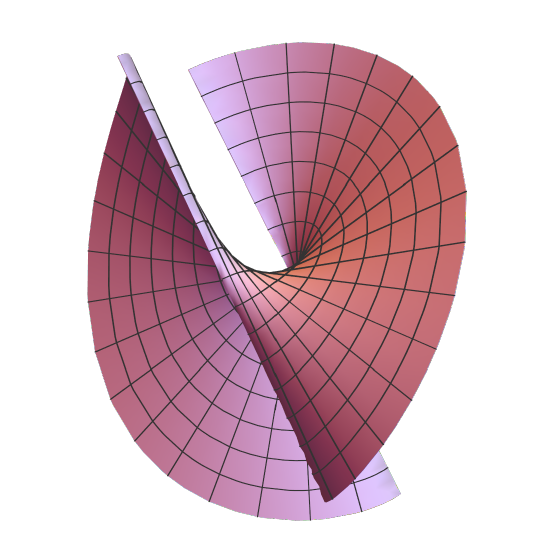}
    \end{subfigure}
    \caption{Spherical curves and non-developable ruled Ricci surfaces for anti-Salkowski curves with $\ell = 0.1$, $\ell = \frac{1}{3}$ and $\ell = 0.57$, respectively.}
    \label{fig:tres-imagens}
\end{figure}
\end{example}

\begin{example}[Borderline spherical curve]

The last ruled Ricci surfaces we present is such that is associated with a curve on the unit sphere whose maximal interval of definition is the real line. Let $B: \Rbb \to \mathbb{S}^2$ be the spherical curve parameterized by arc length as
    \[B(t) = (\tanh(t) \cos(t) , \tanh(t) \sin(t) , \sech(t)).\]
    Notice that $B$ acumulates in the equator as $t \to \pm \infty.$ Direct computations shows that
    \begin{align*}
        B'(t) = \sech^2(t)(\cos(t), \sin(t), 0) + \tanh(t)(-\sin(t), \cos(t), -\sech(t)).
    \end{align*}
    Hence, it follows from formula \eqref{constanttorsion} that the curve
    \[\alpha(t) = (-\cos(t)\sech(t), -\sin(t)\sech(t), \tanh(t) - t)\]
    has constant torsion equals 1 and the map $X: \Rbb^2 \to \Rbb^3,$ given as $X(t,u) = \alpha(t) + uB(t),$ provides the canonical parametrization of a complete ruled Ricci surface.

\begin{figure}[H]
    \begin{subfigure}
        \centering
        \includegraphics[width=0.32\linewidth]{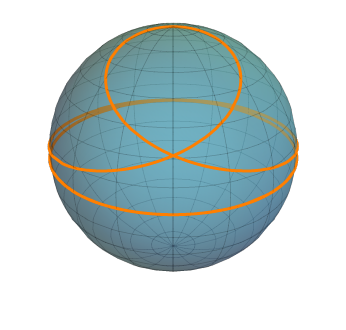}
    \end{subfigure}
    \hspace{2cm}
    \begin{subfigure}
        \centering
        \includegraphics[width=0.31\linewidth]{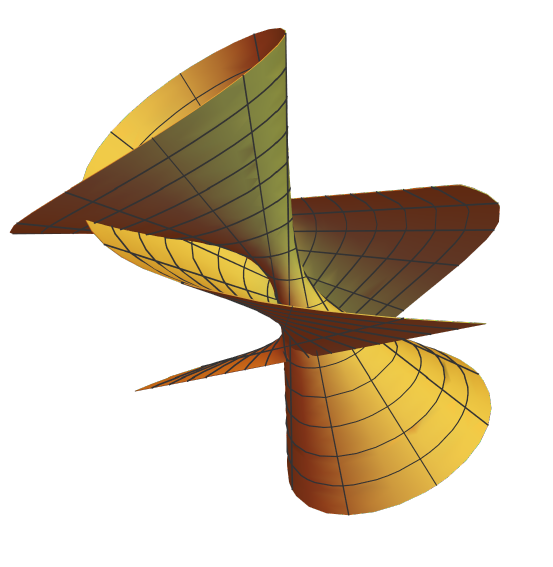}
    \end{subfigure}
    \caption{Borderline spherical curve and its associated non-developable ruled Ricci surface.}
    \label{...}
\end{figure}
\end{example}

\bibliographystyle{amsplain}
\bibliography{references}
\end{document}